\documentclass[11pt, reqno]{amsart}
\usepackage{indentfirst, amssymb, amsmath, amsthm, mathrsfs, setspace, indentfirst, enumerate,  mathrsfs, amsmath, amsthm}
\usepackage[bookmarksnumbered, colorlinks, plainpages]{hyperref}
\usepackage{mathrsfs}
\usepackage[utf8]{inputenc}
\newtheorem*{theoA}{Theorem A}
\newtheorem*{theoB}{Theorem B}
\newtheorem*{theoC}{Theorem C}
\newtheorem*{theoD}{Theorem D}
\newtheorem*{theoE}{Theorem E}

\newtheorem*{conjA}{Conjecture A}
\newtheorem*{exmA}{Example A}
\newtheorem*{remA}{Remark A}

\newtheorem{ques}{Question}[section]

\newtheorem{theo}{Theorem}[section]
\newtheorem{lem}{Lemma}[section]
\newtheorem{cor}{Corollary}[section]

\newtheorem{exm}{Example}[section]

\newtheorem{rem}{Remark}[section]

\newcommand{\ol}{\overline}

\newcommand{\be}{\begin{equation}}
\newcommand{\ee}{\end{equation}}
\newcommand{\beas}{\begin{eqnarray*}}
\newcommand{\eeas}{\end{eqnarray*}}
\newcommand{\bea}{\begin{eqnarray}}
\newcommand{\eea}{\end{eqnarray}}

\numberwithin{equation}{section}
\begin{document}
\title[M\MakeLowercase{eromorphic solutions of} F\MakeLowercase{ermat type differential}...]{\Large M\MakeLowercase{eromorphic solutions of a certain type of nonlinear differential equation}}
\date{}
\author[S. M\MakeLowercase{ajumder}, N. S\MakeLowercase{arkar and} D. P\MakeLowercase{ramanik}]{S\MakeLowercase{ujoy Majumder}, N\MakeLowercase{abadwip} S\MakeLowercase{arkar$^{*}$ and} D\MakeLowercase{ebabrata Pramanik}}
\address{Department of Mathematics, Raiganj University, Raiganj, West Bengal-733134, India.}
\email{sm05math@gmail.com}

\address{Department of Mathematics, Raiganj University, Raiganj, West Bengal-733134, India.}
\email{naba.iitbmath@gmail.com}

\address{Department of Mathematics, Raiganj University, Raiganj, West Bengal-733134, India.}
\email{debumath07@gmail.com}

\renewcommand{\thefootnote}{}
\footnote{*\emph{Corresponding Author}: Nabadwip Sarkar.}
\renewcommand{\thefootnote}{\arabic{footnote}}
\setcounter{footnote}{0}

\begin{abstract} Our paper focuses on investigating the existence and possible forms of solutions to the nonlinear differential equation
\beas f^m+\big(Rf^{(k)}\big)^n=Qe^{\alpha},\eeas
where where $k$, $m$ and $n$ are three positive integers, $Q$ and $R$ are non-zero rational functions, and $\alpha$ is a polynomial.
We examine this equation systematically for all positive integral values of 
$k$, $m$ and $n$, providing comprehensive conditions under which entire or meromorphic solutions may exist. Our results offer substantial improvements over those of Tang and Liao \cite{TL1} and Han and L\"{u} \cite{HL1}, particularly with respect to the existence criteria and the explicit analytic forms of admissible solutions. These findings contribute meaningfully to the broader study of nonlinear differential polynomials and Fermat-type functional equations in complex analysis.
\end{abstract}
\thanks{Typeset by \AmS -\LaTeX}
\maketitle

\section{{\bf Introduction}}
It is known that, Fermat's last theorem states that the equation $x^n+y^n=1$ when $n\geq 3$ does not admit nontrivial rational solutions (see \cite{TW1, W1}). The equation $x^2 + y^2=1$ does admit nontrivial rational solutions. The following functional equation
\bea\label{ne1} f^n(z)+g^n(z)=1,\eea
where $n$ is a positive integer can be regarded as the Fermat diophantine equations $x^n+y^n=1$ over function fields. 
In 1927, Montel \cite{M1} proved that the equation (\ref{ne1}) has no transcendental entire solutions for $n \geq 3$.
Also in Theorem 4 \cite{FG2}, we see that the equation (\ref{ne1}) actually has no non-constant entire solutions when $n>2$. 
Gross \cite[Theorem 3]{FG1} proved that equation (\ref{ne1}) has no non-constant meromorphic solutions when $n>3$. 
For $n=2$, Gross \cite[Theorem 4]{FG2} found that the equation (\ref{ne1}) has entire solutions of the from $f(z)=\sin (h(z))$ and $g(z)=\cos (h(z))$, where $h(z)$ is an entire function.
In Theorem 1, Gross \cite{FG1} proved that all meromorphic solutions of the equation $f^2(z)+g^2(z)=1$ are the form 
\[f(z)=\frac{1-\alpha^2(z)}{1+\alpha^2(z)}\;\;\text{and}\;\;g(z)=\frac{2 \alpha(z)}{1+\alpha^2(z)},\]
where $\alpha(z)$ is a non-constant meromorphic function. For $n=3$, Baker \cite[Theorem 1]{Baker} proved that the only non-constant meromorphic solutions of equation (\ref{ne1}) are the functions
\[f=\frac{1}{2}\left\{1+\frac{\wp^{(1)}(u)}{\sqrt{3}}\right\} / \wp(u)\;\;\text{and}\;\; g=\frac{\eta}{2}\left\{1-\frac{\wp^{(1)}(u)}{\sqrt{3}}\right\} / \wp(u)\]
for a non-constant entire function $u$ and a cubic root $\eta$ of unity, where $\wp$ denotes the Weierstrass function satisfying $(\wp^{(1)})^2=4\wp^3-1$. 

In 1970, Yang \cite{Y1} investigated the following Fermat-type equation
\bea\label{ne2} f^{m}(z)+g^{n}(z)=1\eea
and obtained that equation (\ref{ne2}) has no non-constant entire solutions when $\frac{1}{m}+\frac{1}{n}<1$ (see the proof of Theorem 1). Therefore it is clear that the equation (\ref{ne2}) has no non-constant entire solutions when $m>2$ and $n>2$. However, for the case when $m=n=2$ and $g=f^{(1)}$, many authors investigated the existence of solutions of the equation (\ref{ne2}). As a result, successively several research papers were published (see \cite{GY, HL1, LL, LQ, QY, TL1, WY1, WCH, YL1}).

\smallskip
It is an interesting and quite difficult question to prove that a complex nonlinear differential equation has no meromorphic solution or to find out all the meromorphic solutions of a complex nonlinear differential equation if such solutions exist, even for a simple differential equation. In 2004, Yang and Li \cite{YL1} study for the meromorphic solutions of the following type of the differential equation
\[f^2+b(L(f))^2=a,\]
where $L(f)$ is a linear differential polynomial in $f$, $a$ and $b$ are small meromorphic functions of $f$. Actually they obtained the following results.

\begin{theoA}\cite[Theorem 1]{YL1} Let $n(\geq 1)$ be an integer and let $a, b_0, b_1,\ldots, b_{n-1}$ be polynomials and $b_n$ be a non-zero constant. Let $L(f)=\sum_{k=0}^n b_k f^{(k)}$. If $a\not\equiv 0$, then a transcendental meromorphic solution of the following equation:
\[f^2+(L(f))^2=a,\]
must have the form $f(z)=\frac{1}{2} \left(P(z)e^{R(z)}+Q(z)e^{-R(z)}\right)$, where $P$, $Q$ and $R$ are polynomials and $PQ= a$. If, furthermore, all $b_k$ are constants, then $\deg(P)+\deg(Q) \leq n-1$. Moreover, $R(z)=\lambda$ is a non-zero constant, which satisfies the following equations:
\[\sideset{}{_{k=0}^n}{\sum} b_k \lambda^k=\frac{1}{\iota},\;\;\sideset{}{_{k=j}^n}{\sum} b_k {}^k C_j\lambda^{k-j}=0,\; j=1, \ldots, p,\]
\[\sideset{}{_{k=0}^n}{\sum} b_k(-\lambda)^k=-\frac{1}{\iota},\;\;\sideset{}{_{k=j}^n}{\sum} b_k {}^k C_j(-\lambda)^{k-j} = 0,\;j=1,\ldots,q.\]
\end{theoA}

In Theorem 2 \cite{YL1}, Yang and Li also proved that the differential equation $f^2+b_1^2(f^{(k)})^2=a$ has no transcendental meromorphic solution provided that $b_1(\neq 0)$ is a constant and $a$ is a non-constant meromorphic small function of $f$.
 
\begin{theoB} \cite[Theorem 3]{YL1} Let $a_1$, $a_2$ and $a_3$ be non-zero meromorphic functions. Then a necessary condition for the differential equation
\bea\label{2a}a_1f^2+a_2(f^{(1)})^2=a_3\eea
to have a transcendental meromorphic solution satisfying $T(r,a_k)=S(r,f)$, $k=1,2,3$, is $\frac{a_1}{a_3}\equiv \text{constant}$.
\end{theoB}

Eq. (\ref{2a}) may have some entire solutions, when $a_1/a_3$ is constant. For example, the function $f(z)=\frac{1}{2}(e^{z^2/{2\iota}}+e^{-z^2/{2\iota}})$ satisfies $z^2f^2(z)+(f^{(1)}(z))^2=z^2$.

In the same paper, Yang and Li \cite{YL1} also posed the following conjecture.
\begin{conjA}\cite[Conjecture 1]{YL1} Let $P_1$, $P_2$ and $P_3$ be non-zero polynomials. Then the equation $P_1f^2+
P_2(f^{(1)})^2=P_3$ has no transcendental meromorphic solution when $P_1/P_3$ is a non-zero constant and $P_2/P_3$ is not the square of any rational function.
\end{conjA}

By considering the following example, Tang and Liao \cite{TL1} ensured that Conjecture A is not true in general.
\begin{exmA}\cite{TL1} Let $P_1f^2+P_2(f^{(1)})^2=P_3$, where $P_1\equiv 1$, $P_2(z)=4z$ and $P_3\equiv 1$. Clearly $P_1/P_3$ is a non-zero constant and $P_2/P_3$ is not the square of any rational function. But, $f(z)=\cos(\sqrt{z})$ is a transcendental entire solution of $P_1f^2+P_2(f^{(1)})^2=P_3$.
\end{exmA}

\smallskip
In 2007, Tang and Liao \cite{TL1} first studied the following differential equation
\bea\label{2b} f^2+P^2(f^{(k)})^2=Q,\eea
where $P$ and $Q$ are non-zero polynomials and proved that 

\begin{theoC}\cite[Theorem 1]{TL1} If the Eq. (\ref{2b}) has a transcendental meromorphic solution $f$, then $P\equiv A(\text{constant})$, $Q\equiv B(\text{constant})$, $k=2n+1$ for some non-negative integer $n$ and $f(z)=b\cos(az+c)$, where $a$, $b$ and $c$ are constants such that $Aa^k=\pm 1$ and $b^2=B$.
\end{theoC}

\begin{remA}\cite[Remark 1]{TL1} It follows easily from Theorem C that if either $P$ or $Q$ is a non-constant polynomial or $k$ is an even integer, then the Eq. (\ref{2b}) has no transcendental meromorphic solution.
\end{remA}

\smallskip
Secondly, Tang and Liao \cite{TL1} studied the following differential equation
\bea\label{2c} f^2+\frac{1}{P^2}(f^{(k)})^2=Q,\eea
for the existence of transcendental meromorphic solutions, where $P$ is a non-constant polynomial and $Q$ is a non-zero rational function.
In fact they got the following result.

\begin{theoD}\cite[Theorem 3]{TL1} If the Eq. (\ref{2c}) has a transcendental meromorphic solution $f$, then 
$k$ is an odd integer and $Q$ is a polynomial. Furthermore, if $k=1$, then $Q\equiv C(\text{constant})$ and the solution has the form $f (z)=A\cos (p(z))$, where $A$ is a constant such that $A^2=C$ and $p^{(1)}(z)=\pm P(z)$.
\end{theoD}

\smallskip
In 2019, Han and L\"{u} \cite{HL1} considered the following differential equation
\bea\label{2e}f^n(z)+(f^{(1)}(z))^n=e^{\alpha z+\beta},\eea
where $n(\geq 1)$ is an integer, $\alpha$ and $\beta$ are constants and derived the following result.

\begin{theoE}\cite[Theorem 2.1]{HL1} The meromorphic solutions $f$ of the Eq. (\ref{2e}) must be entire functions and the following assertions hold.
\begin{enumerate}
\item[(A)] For $n=1$, the general solutions of (\ref{2e}) are $f(z)=\frac{e^{\alpha z + \beta}}{\alpha + 1} + a e^{-z}$ for $\alpha \neq -1$, and $f(z) = z e^{-(z + \beta)} + a e^{-z}$.
\item[(B)] For $n=2$, either $\alpha = 0$ and the general solutions of (\ref{2e}) are $f(z) = e^{\frac{\beta}{2}} \sin(z + b)$, or $f(z) = d e^{\frac{\alpha z + \beta}{2}}$.
\item[(C)] For $n \geq 3$, the general solutions of (\ref{2e}) are $f(z) = d e^{\frac{\alpha z + \beta}{n}}$.
\end{enumerate}

Here $a, b, d \in \mathbb{C}$ with $d^n \left(1 + \frac{\alpha}{n}\right)^n = 1$ for $n \geq 1$.
\end{theoE}

\smallskip
In this paper, we assume that the reader is familiar with standard symbols and fundamental results of Nevanlinna Theory (see \cite{WKH,YY1}). Throughout the paper, we denote by $\mu(f)$ and $\rho(f)$ the lower order of $f$ and the order of $f$ respectively. 
As usual, the abbreviation CM means ``counting multiplicities'', while IM means ``ignoring multiplicities''. Let $f$ and $g$ be two non-constant meromorphic functions and $a\in\mathbb{C}$. If $g-a=0$ whenever $f-a=0$, we write $f=a\Rightarrow g=a$.

\medskip
Let $R=\frac{P}{Q}(\not\equiv 0)$ be a rational function, where $P$ and $Q$ are relatively prime polynomials. We define the degree of $R$ as $\deg(R) = \deg(P)-\deg(Q)$. If  $R \equiv 0$, then we define $\deg(R) = -\infty$. If $R_{1}$ and $R_{2}$ are two rational functions, then 
\begin{enumerate}
\item[(i)] $\deg(R_{1}R_{2})=\deg(R_{1})+ \deg(R_{2})$,
\item[(ii)] $\deg\left(\frac{R_{1}}{R_{2}}\right)=\deg(R_{1})-\deg(R_{2})$,
\item[(iii)] $\deg(R^{(1)})\leq \deg(R)-1$,
\item[(iv)] $\deg\left(\frac{R^{(1)}}{R}\right)\leq - 1$,
\item[(v)] $\deg (R_{1} + R_{2}) \leq \max \{\deg(R_{1}), \deg(R_{2})\}$.
\end{enumerate}

\section{\bf{Main results}}

In this section, we consider the following more general differential equation
\bea\label{2.c} f^m+\big(Rf^{(k)}\big)^n=Qe^{\alpha},\eea
where $k$, $m$ and $n$ are three positive integers, $Q$ and $R$ are non-zero rational functions and $\alpha$ is a polynomial.

\smallskip
It is easy to verify that the Eq. (\ref{2.c}) has no rational solution provided that $\alpha$ is non-constant.  If $m=n=1$, it is easy to verify that $f(z)=e^{z^2}$ is a solution of $f(z)+R(z)f^{(k)}(z)=Q(z)e^{\alpha(z)}$, where $k=2$, $R(z)=z$, $Q(z)=1+2z+4z^3$ and $\alpha(z)=z^2$. Therefore we study the Eq. (\ref{2.c}) for the existence of solutions for the case when $m+n>2$.

\smallskip
The following question is natural to pose:
\begin{ques}\label{ques1} What can be concluded about the existence of non-constant meromorphic solution $f$ of the Eq. (\ref{2.c}) when $m+n>2$?
\end{ques}

Our first objective is to provide a complete solution to Question \ref{ques1}. But, here we give a partial answer of Question \ref{ques1} under some conditions.
Now we state our first result.
 
\begin{theo}\label{t2.1} Let $R(\not\equiv 0)$ and $Q(\not\equiv 0)$ be a rational functions, $\alpha$ be polynomials and let $k$, $m$ and $n$ be three positive integers such that $m\neq n$ and $m+n>2$. Let $f$ be any transcendental meromorphic function such that one of the following holds:
\begin{enumerate}
\item[(i)] $\rho(f)=+\infty$;
\item[(ii)] $\rho(f)\leq +\infty$, when $\deg(\alpha)=0$;
\item[(iii)] $\deg(\alpha)<\mu(f)$, when $\deg(\alpha)>0$.
\end{enumerate}

Then $f$ can not be a solution of the Eq. (\ref{2.c}),
\begin{enumerate}
\item[(A)] if either $m>n>1$ or $n>m>1$;
\item[(B)] if either $m>n=1$ or $n>m=1$ and $n\geq k+1$.
\end{enumerate}
\end{theo}

When $\alpha$ is a constant, we obtain the following corollary from Theorem \ref{t2.1}.

\begin{cor}\label{c2.1} Let $R(\not\equiv 0)$ and $Q(\not\equiv 0)$ be a rational functions and let $k$, $m$ and $n$ be three positive integers such that $m\neq n$ and $m+n>2$. If the following Eq.
\bea\label{2.ca} f^m+\big(Rf^{(k)}\big)^n=Q,\eea
satisfies one of the following conditions:
\begin{enumerate}
\item[(i)] $m>n>1$ or $n>m>1$;
\item[(ii)] $m>n=1$ or $n>m=1$ and $n\geq k+1$,
\end{enumerate}
then the Eq. (\ref{2.ca}) does not have any transcendental meromorphic solution.
\end{cor}

If $m=n$, then the Eq. (\ref{2.c}) reduces to
\bea\label{2.d} f^m+\big(Rf^{(k)}\big)^m=Qe^{\alpha},\eea
where $k(\geq 1)$ and $m(\geq 2)$ are two integers, $Q$ and $R$ are non-zero rational functions and $\alpha$ is a polynomial.
We now state our next result.

\begin{theo}\label{t2.2} Let $Q(\not\equiv 0)$ and $R(\not\equiv 0)$ be rational functions, $\alpha$ be a polynomial and let $k(\geq 1)$ and $m(\geq 2)$ be integers. Then the Eq. (\ref{2.d}) has no meromorphic solution of infinite order.
\end{theo}

According to Theorem \ref{t2.2}, every solutions of the Eq. (\ref{2.d}) must be of finite order.

\smallskip
Our second objective is to explore all possible characterizations of solutions to the Fermat-type functional equation (\ref{2.c}) depending on the nature of $m(\geq 2)$, $R$, $Q$ and $\alpha$.

\medskip
When $R(\not\equiv 0)$ is a polynomial and $Q\equiv 1$, we obtain the following result which includes complete list of solutions of the Eq. (\ref{2.d}) for different values of $m(\geq 2)$.
 
\begin{theo}\label{t2.3} Let $R(\not\equiv 0)$ and $\alpha$ be polynomials and let $k(\geq 1)$ and $m(\geq 2)$ be two integers. 
The meromorphic solutions $f$ of the following Eq. 
\bea\label{2.e} f^m+\big(Rf^{(k)}\big)^m=e^{\alpha},\eea
must be entire functions, $R$ reduces to a non-zero constant and the following assertions hold.
\begin{enumerate}
\item[(A)] For $m=2$, the general solutions of (\ref{2.e}) are one of the following:
\begin{enumerate}
\item[(A1)] $f(z)=\frac{d^2-1}{2\iota d}e^{\frac{az+b}{2}}$, where $a(\neq 0), b, d(\neq 0)$ are constants such that $d^2(R(a/2)^k-\iota)=R(a/2)^k+\iota$;

\medskip
\item[(A2)] $f(z)=e^{\frac{a_1z+a_2}{2}}\sin (b_1z+b_2)$, where $a_1,a_2, b_1(\neq 0),b_2$ are constants such that $R\left(\frac{a_1}{2}+\iota b_1\right)^k=\iota$ and $R\left(\frac{a_1}{2}-\iota b_1\right)^k=-\iota$. Furthermore if $a_1=0$, then $k$ is odd;
\end{enumerate}

\medskip
\item[(B)] For $m\geq 3$, the general solutions of (\ref{2.e}) are $f(z)=ce^{\frac{az+b}{m}}$, where $a(\neq 0), b, c(\neq 0)$ are constants such that $c^m\left(R^m(a/m)^{km}+1\right)=1$.
\end{enumerate}
\end{theo}

\begin{rem} Obviously the scope of Theorem E is extended in Theorem \ref{t2.3} by considering a larger class of functional equation (\ref{2.e}) including the case when $k=1$, $R\equiv 1$ and $\alpha(z)=az+b$ which leads to the specific type of functional equation (\ref{2e}). Consequently Theorem \ref{t2.3} improves Theorem E. Also Theorem \ref{t2.3} shows that the Eq. (\ref{2.e}) has no non-constant rational solutions.
\end{rem}

\medskip
Next, we intend to explore all possible characterizations of solutions to the Eq. (\ref{2.d}) when $Q(\not\equiv 0)$ and $R(\not\equiv 0)$ are rational functions. Consequently, we obtain the following result.

\begin{theo}\label{t2.4} Let $Q(\not\equiv 0)$ and $R(\not\equiv 0)$ be rational functions, $\alpha$ be a polynomial and let $k(\geq 1)$ and $m(\geq 2)$ be two integers. Then the transcendental meromorphic solution of the Eq. (\ref{2.d}) are characterized as follows:
\begin{enumerate}
\item[(A)] $f=R_1e^{\alpha/m}$, where $\alpha$ is non-constant and $R_1(\not\equiv 0)$ is a rational function such that $mk\deg(\alpha^{(1)})=\deg(Q-R_1^m)-m\deg(RR_1)$ and $R_1^m+R^m e^{-\alpha}\left((R_1e^{\alpha/m})^{(k)}\right)^m=Q$;

\medskip
\item[(B)] $f=\frac{d^2Q_1-Q_2}{2\iota d}e^{\alpha/2}$,
where $Q_1$ and $Q_2$ are rational functions and $d(\neq 0)$ is a constant such that $Q_1Q_2=Q$ and $d^2Q_1-Q_2\not\equiv 0$ and $\alpha$ is non-constant;

\medskip
\item[(C)] $f(z)=\frac{Q_1e^{a_1z+b_1}-Q_2e^{a_2z+b_2}}{2\iota}$,
where $R\equiv A(\text{constant})$ such that $Aa_1^k=\iota$, $Aa_2^k=-\iota$, $Q_1$ and $Q_2$ are constants such that $Q_1Q_2=Q$, $b_1$ and $b_2$ are constants such that $\alpha(z)=(a_1+a_2)z+b_1+b_2$. Furthermore if $\alpha$ is constant, i.e., if $a_1+a_2=0$, then $k$ is odd;

\medskip
\item[(D)] $f(z)=\frac{Q_1(z)e^{a_1z+b_1}-Q_2(z)e^{a_2z+b_2}}{2\iota}$,
where $Q\equiv B(\text{constant})$, $Q_1$, $Q_2$ and $R$ are non-constant rational functions such that $Q_1Q_2=B$, $\deg(R)=0$ and $a_1(\neq 0), a_2(\neq 0)$, $b_1$ and $b_2$ are constants such that $a_1^k+a_2^k=0$ and $\alpha(z)=(a_1+a_2)z+b_1+b_2$. Furthermore if $\alpha$ is a constant, then $k$ is odd;

\medskip
\item[(E)] $f(z)=\frac{Q_1(z)e^{tP(z)+c}-Q_2(z)e^{P(z)}}{2\iota}$,
where $P$ is a non-constant polynomial, $Q_1$, $Q_2$ and $R$ are non-constant rational functions such that $k\deg\left(P^{(1)}\right)=-\deg(R)$, $Q_1Q_2=Q$, $\deg(R)<0$, $c$ and $t$ are constants such that $t^k=-1$ and $(t+1)P^{(1)}=\alpha^{(1)}$. If $\alpha$ is a constant, then $t=-1$ and $k$ is odd.
Furthermore if all the zeros of $R$ have multiplicities at most $k-1$, then $Q_1$, $Q_2$ and $Q$ are all polynomials. If $k=1$, then $Q_1$, $Q_2$ and $Q$ are all non-zero constants and $R=\frac{1}{P^{(1)}}$.

\end{enumerate}
\end{theo}

The following examples show the existence of transcendental meromorphic solutions of the Eq. (\ref{2.d}).

\begin{exm} Let $m=2$, $k=1$, 
\[R(z)=\frac{z^2+1}{z+1},\;Q(z)=\left(\frac{z-1}{z+1}\right)^2+\left(\frac{(z^2+1)(z^2+3)}{2(z+1)^3}\right)^2\;\;\text{and}\;\;\alpha(z)=z.\]

If $R_1(z)=\frac{z-1}{z+1}$, then it is easy to verify that $km\deg(\alpha^{(1)})=\deg(Q-R_1^m)-m\deg(RR_1)$ and $R_1^m+R^m e^{-\alpha}\left((R_1e^{\alpha/m})^{(k)}\right)^m=Q$. Also it is easy to check that $f=R_1e^{\frac{\alpha}{m}}$ is a solution of $f^m+(Rf^{(k)})^m=Qe^{\alpha}$.
\end{exm}

\begin{exm} Let $m=3$, $k=1$, 
\[R(z)=\frac{1}{z},\;Q(z)=\frac{z^9+3z^8+3z^7+4z^6+11z^5+3z^4+4z^3+1}{z^3}\;\;\text{and}\;\;\alpha(z)=z^3.\]

If $R_1(z)=z+1$, then it is easy to verify that $km\deg(\alpha^{(1)})=\deg(Q-R_1^m)-m\deg(RR_1)$ and $R_1^m+R^m e^{-\alpha}\left((R_1e^{\alpha/m})^{(k)}\right)^m=Q$. Also it is easy to check that $f=R_1e^{\frac{\alpha}{m}}$ is a solution of $f^m+(Rf^{(k)})^m=Qe^{\alpha}$.
\end{exm}

\smallskip
When $\alpha$ is a constant, we get the following corollary from Theorem \ref{t2.4}.

\begin{cor}\label{c2.2} Let $Q(\not\equiv 0)$ and $R(\not\equiv 0)$ be rational functions and let $k(\geq 1)$ and $m(\geq 2)$ be two integers. Then the transcendental meromorphic solution of the following Eq.
\[f^m+\big(Rf^{(k)}\big)^m=Q\]
are characterized as follows:
\begin{enumerate}
\item[(A)] $f(z)=\frac{Q_1e^{a_1z+b_1}-Q_2e^{-a_1z+b_2}}{2\iota}$,
where $k$ is odd, $R\equiv A(\text{constant})$ such that $Aa_1^k=\iota$, $Q_1$ and $Q_2$ are constants such that $Q_1Q_2=Q$, $b_1$ and $b_2$ are constants;

\medskip
\item[(B)] $f(z)=\frac{Q_1(z)e^{a_1z+b_1}-Q_2(z)e^{-a_1z+b_2}}{2\iota}$,
where $k$ is odd, $Q\equiv B(\text{constant})$, $Q_1$, $Q_2$ and $R$ are non-constant rational functions such that $Q_1Q_2=B$, $\deg(R)=0$ and $a_1(\neq 0)$, $b_1$ and $b_2$ are constants;

\medskip
\item[(C)] $f(z)=\frac{Q_1(z)e^{tP(z)+c}-Q_2(z)e^{P(z)}}{2\iota}$,
where $k$ is odd, $P$ is a non-constant polynomial, $Q_1$, $Q_2$ and $R$ are non-constant rational functions such that $k\deg\left(P^{(1)}\right)=-\deg(R)$, $Q_1Q_2=Q$, $\deg(R)<0$ and $c$ is a constant.
Furthermore if all the zeros of $R$ have multiplicities at most $k-1$, then $Q_1$, $Q_2$ and $Q$ are all polynomials. If $k=1$, then $Q_1$, $Q_2$ and $Q$ are all non-zero constants and $R=\frac{1}{P^{(1)}}$.
\end{enumerate}
\end{cor}

\smallskip
\begin{rem} Regarding Corollary \ref{c2.2}, we observe that the scope of Theorems C and D were extended in Corollary \ref{c2.2} by considering a larger class of functional equations including the case when
\end{rem}

\begin{enumerate}
\item[(i)] $k=1$, $m=2$, $R=P$ and $Q$ are non-zero polynomials;
\item[(ii)] $k=1$, $m=2$, $R=1/P$ and $Q(\not\equiv 0)$ are rational functions, where $P$ is a non-zero polynomial
\end{enumerate}
which leads to the specific type of Fermat-type functional equation.

\medskip
From the Eq. (\ref{2.c}), it is easy to verify that $\deg(\alpha)\leq \mu(f)$. Therefore for our further study, we raise the following questions.

\begin{ques}  What can be concluded about the existence of transcendental meromorphic solution $f$ of the Eq. (\ref{2.c}) when $0<\deg(\alpha)=\mu(f)$?
\end{ques}

\begin{ques} What can be concluded about the existence of transcendental meromorphic solutions of the Eq. (\ref{2.ca}) when $n>m=1$ and $n\leq k$?
\end{ques}

\section{\bf{Some Lemmas}} 

\begin{lem}\label{l2}\cite{MK1} Let $f$ be a non-constant meromorphic function, $P(f)=\sum_{k=0}^{p} a_{k}f^{k}$ and $Q(f)=\sum_{j=0}^{q} b_{j}f^{j}$ be two mutually prime polynomials in $f$, where $a_{k}$ and $b_{j}$ are small functions of $f$ such that $a_{p}b_{q}\not\equiv 0$. If $R(f)=\frac{P(f)}{Q(f)}$, then $T(r,R(f))=\max\{p,q\}\;T(r,f)+S(r,f).$
\end{lem}

\begin{lem}\label{l3}\cite[Lemma 2.4]{ZY1} Let $f$ be a non-constant meromorphic function and let $k$ and $p$ be positive integers. Then
\beas \label{2} N_{p}\left(r,0;f^{(k)}\right) \leq T\left(r,f^{(k)}\right)-T(r,f)+ N_{p+k}(r,0;f) + S(r,f).\eeas
\end{lem}

\begin{lem}\label{l3.1}\cite[Lemma 2]{JC1} Let $f$ be a transcendental meromorphic function such that $f^{n}P(f)=Q(f)$, where $P(f)$ and $Q(f)$ are differential polynomials in $f$ with functions of small proximity related to $f$ as the coefficients and the degree of $Q(f)$ is at most $n$. Then $m(r,P)=S(r,f).$ 
\end{lem}

\begin{lem} \label{l3.2}\cite[Lemma 3.5]{WKH} Let $F$ be meromorphic in a domain $D$ and $n$ be a positive integer. Then 
\beas \frac{F^{(n)}}{F}=f^{n}+\frac{n(n-1)}{2}f^{n-2}f^{(1)}+a_{n}f^{n-3}f^{(2)}+b_{n}f^{n-4}(f^{(1)})^{2}+P_{n-3}(f),\eeas 
where $f=F^{(1)}/F$, $a_{n}=\frac{1}{6}n(n-1)(n-2)$, $b_{n}=\frac{1}{8}n(n-1)(n-2)(n-3)$ and $P_{n-3}(f)$ is a differential polynomial with constant coefficients, which vanishes identically for $n\leq 3$ and has degree $n-3$ when $n>3$. 
\end{lem}

In order to state the following lemma, we need the notation $M(r, f)$ to denote
the maximum of an entire function $f$, i.e., $M(r,f)=\max\limits_{|z|=r}|f(z)|$.

Suppose the Taylor expansion of entire function $f$ is $f(z)=\sum_{n=0}^{\infty} a_{n}z^{n}$. We denote by \[\mu (r,f)=\max \limits_{n\in\mathbb{N},\; |z|=r}\{|a_{n}z^{n}|\}\;\text{and}\;\nu(r,f)=\sup\{n:|a_{n}|r^{n}=\mu (r,f)\}.\]

Here it is enough to recall that (see \cite{JV1})
\begin{enumerate}
\item[(1)] $\nu(r,f)$ is increasing, piecewise constant, right-continuous and also tends to $+\infty$ as $r\rightarrow \infty$;
\item[(2)] $\log \nu(r,F)=O(\log r)$, if $\rho(f)<+\infty$.
\end{enumerate}

\begin{lem}\label{l3.3} \cite[Satz 21.3]{JV1} Let $f$ be a transcendental entire function and let $0<\delta<1/4$ and $z$ be such that $|z|=r$ and that $|f(z)|>M(r,f) (\nu(r,f))^{-\frac{1}{4}+\delta}$ holds. Then there exists a set $E\subset (0,+\infty)$ with finite logarithmic measure such that 
\beas \frac{f^{(m)}(z)}{f(z)}=\left(\frac{\nu(r,f)}{z}\right)^{m}(1+o(1))\eeas 
holds for all $m\geq 0$ and for all $r\not\in E$.
\end{lem}

\begin{lem}\label{l3} Let $P_1(\not\equiv 0)$, $P_2(\not\equiv 0)$ and $Q(\not\equiv 0)$ be rational functions, $\alpha$ be a polynomial and let $k$, $m$ and $n$ be positive integers. If $f$ be a transcendental meromorphic solution of the following differential equation
\bea\label{lem} (P_1f)^m+\left(P_2f^{(k)}\right)^n=Qe^{\alpha},\eea
such that $P_1e^{-\alpha/m}f$ is transcendental, then
\bea\label{lem1} \frac{1}{m}+\frac{1}{n}\geq 1.\eea
\end{lem}

\begin{proof} Let $f$ be a transcendental meromorphic solution of the Eq. (\ref{lem}).
It is easy to verify from (\ref{lem}) that $f$ has only finitely many poles. Again from (\ref{lem}), we get
\bea\label{lem3} F^m+G^n=Q,\eea
where $F=P_1e^{-\alpha/m}f$ and $G=P_2e^{-\alpha/n}f^{(k)}$. By the given condition $F$ is transcendental and so from (\ref{lem3}), we deduce that $G$ is also transcendental. Clearly both $F$ and $G$ have only finitely many poles. Using Lemma \ref{l2} to (\ref{lem3}), we get
\bea\label{lem4} mT(r,F)+S(r,F)=nT(r,G)+S(r,G).\eea

Let
\bea\label{lem5} h=\frac{F^{m}-Q}{F^{m}}.\eea

Clearly $h$ is a non-constant. Again using Lemma \ref{l2} to (\ref{lem5}), we get 
\[T(r,h)+o(T(r,h))=mT(r,F)+S(r,F).\]

Now using (\ref{lem3}) to (\ref{lem5}), we get
\[\ol N(r,h)\leq \ol N(r,0, F^{m})\leq N(r,Q)+\ol N(r,0,F^{m})=\ol N(r,0,F)+S(r,F),\]
\[\ol N(r,0,h)=\ol N(r,Q,F^{m})\leq \ol N(r,0,G^{n})\leq \ol N(r,0,G)+S(r,G)\]
and 
\[\ol N(r,1,h)\leq \ol N(r,F^{m})+\ol N(r,Q)=\ol N(r,f)+S(r,F)=S(r,F).\] 

Therefore by second fundamental theorem, we get
\beas mT(r,F)=T(r,h)+S(r,h)&\leq& \ol N(r,h)+\ol N(r,0,h)+\ol N(r,1,h)+S(r,h)\nonumber\\&\leq&
\ol N(r,0,F)+\ol N(r,0,G)+S(r,F)+S(r,G)\nonumber\\&\leq&
T(r,F)+T(r,G)+S(r,F)+S(r,G)\nonumber
\eeas
and so from (\ref{lem4}), we obtain $\left(m-1-\frac{m}{n}\right)T(r,F)\leq S(r,F)$,
which shows that $\frac{1}{m}+\frac{1}{n}\geq 1$.  
\end{proof}

\section {{\bf Proof of Theorem \ref{t2.1}}} 
\begin{proof} Let $f$ be a transcendental meromorphic solution of the Eq. (\ref{2.c}) such that one of the following holds:
\begin{enumerate}
\item[(i)] $\rho(f)=+\infty$;
\item[(ii)] $\rho(f)\leq +\infty$, when $\deg(\alpha)=0$;
\item[(iii)] $\deg(\alpha)<\mu(f)$, when $\deg(\alpha)>0$.
\end{enumerate}

 Note that if $\rho(f)=+\infty$, then obviously $T(r,e^{\alpha})=S(r,f)$. Again if $0<\deg(\alpha)<\mu(f)$, then we have $0<\deg(\alpha)=\rho(e^{\alpha})<\mu(f)$ and so by Theorem 1.18 \cite{YY1}, we get $T(r,e^{\alpha})=S(r,f)$. Thus in either case, we have $T(r,e^{\alpha})=S(r,f)$. Also from (\ref{2.c}) it is easy to verify that $N(r,f)=O(\log r)$, i.e., $f$ has only finitely many poles. Now we consider the following cases.
 
\smallskip
{\bf Case 1.} Let either $n>m>1$ or $m>n>1$. Then from (\ref{lem1}), we get a contradiction. Hence in these cases, $f$ can not be a solution of the equation (\ref{2.c}).

\smallskip
{\bf Case 2.} Let $m>n=1$. Then from (\ref{2.c}), we have
\bea\label{s.1} f^m+Rf^{(k)}=Qe^{\alpha}.\eea

Since $N(r,f)=O(\log r)$, using Lemma \ref{l2} to (\ref{s.1}), we get
\beas m T(r,f)=T\left(r, f^{m}\right)+S(r,f)
& =&T\left(r,Rf^{(k)}-Qe^{\alpha}\right)+S(r,f) \nonumber\\
&\leq & T\left(r,f^{(k)}\right)+S(r,f) \nonumber\\
&\leq & m\left(r, \frac{f^{(k)}}{f}\right)+m(r,f)+S(r,f)\\
&\leq& T(r, f)+S(r, f),\eeas
which is impossible.

\smallskip
{\bf Case 3.} Let $n>m=1$ and $n\geq k+1$. Then from (\ref{2.c}), we have
\bea\label{s.2} f+\left(Rf^{(k)}\right)^n=Qe^{\alpha}.\eea

Differentiating (\ref{s.2}) once, we get $n(Rf^{(k)})^{n-1}(Rf^{(k)})^{(1)}=-(f-Qe^{\alpha})^{(1)}$ and so
\bea\label{s.3} n^n\left(f-Qe^{\alpha}\right)^{n-1}\left(\left(Rf^{(k)}\right)^{(1)}\right)^n=-\left((f-Qe^{\alpha})^{(1)}\right)^n.\eea

Let $z_0$ be a zero of $f-Qe^{\alpha}$ of multiplicity $p_0$ such that $R(z_0)\neq \infty$. Obviously $z_0$ is a zero of $(f-Qe^{\alpha})^{(1)}$ of multiplicity $p_0-1$. If $p_0<n$, then from (\ref{s.3}), we immediately get a contradiction. Hence $p_0\geq n\geq k+1$. Clearly $z_0$ is a zero of  $f^{(k)}-(Qe^{\alpha})^{(k)}$ and so $f-Qe^{\alpha}=0\Rightarrow f^{(k)}-(Qe^{\alpha})^{(k)}=0$ except for the poles of $R$. Again from (\ref{s.2}), we see that 
$f-Qe^{\alpha}$ and $Rf^{(k)}$ share $0$ IM.
Now we consider the following two sub-cases:

\smallskip
{\bf Sub-case 3.1.} Let $(Qe^{\alpha})^{(k)}\not\equiv 0$. Suppose $(Qe^{\alpha})^{(k)}=\tilde Q e^{\alpha}$, where $\tilde Q(\not\equiv 0)$ is a rational function. Since $f-Qe^{\alpha}=0\Rightarrow f^{(k)}-(Qe^{\alpha})^{(k)}=0$ except for the poles of $R$ and $f-Qe^{\alpha}$ and $Rf^{(k)}$ share $0$ IM, it follows that $f-Qe^{\alpha}=0\Rightarrow \tilde Q=0$ except for the poles of $R$ and so $f-Qe^{\alpha}$ has only finitely many zeros. Consequently $f^{(k)}$
has only finitely many zeros. Let $g=f-Qe^{\alpha}$. Then by Lemma \ref{l3}, we have 
\beas T(r,g)\leq T(r,g^{(k)})+N_{k+1}(r,0;g)-\ol N(r,0;g^{(k)})+S(r,g)\eeas
and so by second fundamental theorem for small function (see \cite{KY1}) and Lemma \ref{l2}, we get
\beas T(r,f)&\leq& T(r,f^{(k)})+N_{k+1}(r,0;g)-\ol N(r,0;g^{(k)})+S(r,f)\\&\leq&
\ol N(r,f^{(k)})+\ol N(r,0;f^{(k)})+\ol N(r,(Qe^{\alpha})^{(k)};f^{(k)})+N(r,Qe^{\alpha};f)\\&&-\ol N(r,(Qe^{\alpha})^{(k)};f^{(k)})+S(r,f)\\&\leq& O(\log r)+S(r,f)=S(r,f),\eeas
which is impossible.

\smallskip
{\bf Sub-case 3.2.} Let $(Qe^{\alpha})^{(k)}\equiv 0$. Clearly $\alpha$ is a constant and $Q(\not\equiv 0)$ is a polynomial of degree atmost $k-1$. Let $g=f-Qe^{\alpha}$. Then $g^{(k)}=f^{(k)}$ and so from (\ref{s.2}), we get
\bea\label{s.4} \left(Rg^{(k)}\right)^n=-g.\eea

\smallskip
First we suppose that $k=1$. Differentiating (\ref{s.4}) once, we get
\bea\label{s.5} n R^{n-1}(g^{(1)})^{n-1}\left(R^{(1)}g^{(1)}+Rg^{(2)}\right)=-g^{(1)}.\eea

Note that $g$ is transcendental and so $Rg^{(1)}$ is also transcendental.
Since $n\geq 2$, using Lemma \ref{l3.1} to (\ref{s.5}), we get $m(r,R^{(1)}g^{(1)}+Rg^{(2)})=S(r,g^{(1)})$, i.e., $T(r,R^{(1)}g^{(1)}+Rg^{(2)})=S(r,g^{(1)})$, i.e., $T(r,(Rg^{(1)})^{(1)})=S(r,Rg^{(1)})$, which contradicts the fact that $Rg^{(1)}$ is transcendental.

\smallskip
Next we suppose that $k\geq 2$. Now we divide following two sub-cases:

\smallskip
{\bf Sub-case 3.2.1.} Let $R$ be rational function having no zeros. Then from (\ref{s.4}), it is easy to deduce that $g$ is a transcendental entire function and so $M(r,g)\rightarrow \infty$ as $r\rightarrow \infty$. Again let $M(r,g)=|g(z)|$, where $z_r=r\exp(i\theta)$ and $\theta\in [0,2\pi)$.
Consequently the inequality $|g(z_r)|>M(r,g) (\nu(r,g))^{-\frac{1}{4}+\delta}$ holds when
$0<\delta<1/4$. Now by Lemma \ref{l3.3}, there exists a subset $E\subset [1,+\infty)$ with finite logarithmic measure such that for $z_r=r\exp(i \theta)(\theta\in [0,2\pi))$ satisfying $|z_r|=r\not\in E$ and $M(r,g)=|g(z_r)|$, we get
\bea\label{s.4a} \frac{g^{(k)}(z_r)}{g(z_r)}=\left(\frac{\nu(r,g)}{z_r}\right)^{k}(1+o(1))\;\text{as}\;r\rightarrow \infty.\eea

On the other hand, from (\ref{s.4}), we get
\bea\label{s.4b} \left(\frac{g^{(k)}(z_r)}{g(z_r)}\right)^n=-\frac{\frac{1}{R^n(z_r)}}{g^{n-1}(z_r)}.\eea

Therefore using (\ref{s.4b}) to (\ref{s.4a}), we obtain
\bea\label{s.4c} \nu(r,g)^{nk}(1+o(1))=-\frac{\frac{z_r^{nk}}{R^n(z_r)}}{g^{n-1}(z_r)}.\eea

Note that $\frac{z^{nk}}{R^n(z)}$ is a polynomial. Since $M(r, g^{n-1})$ increases faster than the maximum modulus of any polynomial, from (\ref{s.4c}), it is easy to deduce that $\nu(r,g)^{nk}(1+o(1))$ tends to $0$ as $r\to\infty$, which contradicts the fact that $\nu(r,g)$ tends to $+\infty$ as $r\to\infty$.

\smallskip
{\bf Sub-case 3.2.2.} Let $R$ be rational function having atleast one zero. Clearly $R^{(1)}\not\equiv 0$.

If possible, suppose that $g$ has only finitely many zeros. Also we know that $g$ has only finitely many poles. Now from (\ref{s.4}), we get 
\[-\frac{1}{R^n g^{n-1}}=\left(\frac{g^{(k)}}{g}\right)^n\]
and so using the first fundamental theorem and Lemma \ref{l2}, we obtain
$(n-1)T(r,g)+S(r,g)=S(r,g)$. Since $n\geq k+1$, we get a contradiction. Hence $g$ has infinitely many zeros.
Let $z_0$ be a zero of $g$ of multiplicity $p_0$ such that $R(z_0)\neq 0,\infty$. Then from above, we know that $p_0\geq n\geq k+1$ and so $z_0$ is also a zero of $g^{(k)}$ of multiplicity $p_0-k\geq 1$. Clearly $z_0$ is a zero of $(Rg^{(k)})^n$ of multiplicity $np_0-nk$ and so from (\ref{s.4}), we get $np_0-nk=p_0$, i.e., $p_0(n-1)=nk$. If $p_0\geq n+1$, then since $n\geq k+1$, we have $nk\geq (n+1)(n-1)\geq (n+1)k$, which is impossible. Hence $p_0=n$ and so $n=k+1$. Therefore from (\ref{s.4}), we get
\bea\label{s.6} (Rg^{(k)})^{k+1}=-g.\eea

Clearly from (\ref{s.6}), we may assume that $g=h^{k+1}$, where $h$ is a transcendental meromorphic function having finitely many poles.
Since $g$ has infinitely many zeros, it follows that $h$ has also infinitely many zeros. Therefore from (\ref{s.6}), we get
\bea\label{s.7} R(h^{k+1})^{(k)}=th,\eea
where $t^{k+1}=-1$. 
Now from the proof of Theorem 1.1 \cite{Mss}, we can find that
\bea\label{s.8} (h^{k+1})^{(k)}&=&
(k+1)!h(h^{(1)})^{k}+\frac{k(k-1)}{4}(k+1)!h^{2}(h^{(1)})^{k-2}h^{(2)}+\cdots+(k+1)h^{k}h^{(k)}\nonumber\\&=&
(k+1)!h(h^{(1)})^{k}+\frac{k(k-1)}{4}(k + 1)!h^{2}(h^{(1)})^{k-2}h^{(2)}+R_1(h),\eea
where $R_{1}(h)$ is a differential polynomial in $h$ and each term of $R_{1}(h)$ contains $h^{m}\;(3 \leq m \leq k$) as a factor. Therefore using (\ref{s.8}) to (\ref{s.7}), we get
\bea\label{s.9} R\left((k+1)!(h^{(1)})^{k}+\frac{k(k-1)}{4}(k+1)!h(h^{(1)})^{k-2}h^{(2)}+\cdots+(k+1)h^{k-1}h^{(k)}\right)=t.\eea

Differentiating (\ref{s.9}) once, we obtain
\bea\label{s.10} (k+1)!R^{(1)}(h^{(1)})^{k}+(k+1)!\frac{k(k+3)}{4}R(h^{(1)})^{k-1}h^{(2)}+R_2(h)=0,\eea
where $R_{2}(h)$ each term of $R_{2}(h)$ contains $h^{m}\;(1 \leq m \leq k$) as a factor. 
Let $z_1$ be a zero of $h$ of multiplicity $p_1$. If $p_1\geq 2$, then from (\ref{s.9}), we see that $z_1$ must be a pole of $R$ and so $h$ has finitely many multiple zeros. Let $z_1$ be a simple zero of $h$ such that $R(z_1)\neq 0,\infty$ and $R^{(1)}(z_1)\neq 0$. Then from (\ref{s.9}) and (\ref{s.10}), we have respectively 
\[(k+1)!R(z_1)(h^{(1)}(z_1))^{k}=t\] 
and
\[(k+1)!R^{(1)}(z_1)(h^{(1)}(z_1))^{k}+(k+1)!\frac{k(k+3)}{4}R(z_1)(h^{(1)}(z_1))^{k-1}h^{(2)}(z_1)=0.\]

Then a simple computation shows that 
\[R^{(1)}(z_1)h^{(1)}(z_1)+\frac{k(k+3)}{4} R(z_1)h^{(2)}(z_1)=0\]
and so $h=0\Rightarrow R^{(1)}h^{(1)}+\frac{k(k+3)}{4}Rh^{(2)}=0$. Let us define 
\bea\label{s.11} \beta=\frac{R^{(1)}h^{(1)}+\frac{k(k+3)}{4}Rh^{(2)}}{h}.\eea

If possible suppose $R^{(1)}h^{(1)}+\frac{k(k+3)}{4}Rh^{(2)}\equiv 0$. Then $\frac{h^{(2)}}{h^{(1)}}=-\frac{4}{k(k+3)}\frac{R^{(1)}}{R}$ and so on integration, we have $h^{(1)}=dR^{-\frac{4}{k(k+3)}}$, where $d$ is a non-zero constant. Since $h$ is transcendental, we get a contradiction. Hence $R^{(1)}h^{(1)}+\frac{k(k+3)}{4}Rh^{(2)}\not\equiv 0$ and so $\beta\not\equiv 0$. Since $h$ has only finitely many multiple zeros and $h=0\Rightarrow R^{(1)}h^{(1)}+\frac{k(k+3)}{4}Rh^{(2)}=0$, it is easy to see that $N(r,\beta)=S(r,h)$. Also  $m(r,\beta)=S(r,h)$. Therefore $T(r,\beta)=S(r,h)$.
Now (\ref{s.11}) gives
\bea\label{s.12} h^{(2)}=\frac{4}{k(k+3)R}(R^{(1)}h^{(1)}-\beta h).\eea

Therefore using (\ref{s.12}) to (\ref{s.10}), we get 
\bea\label{s.13}(k+1)!R(h^{(1)})^k+\frac{(k-1)(k+1)!}{(k+3)}\left(R^{(1)}h(h^{(1)})^k-\beta h^2(h^{(1)})^{k-2}\right)+R_3(h)=t,\eea
where each term of $R_{3}(h)$ contains $h^{m}\;(2\leq m \leq k$) as a factor. 
Differentiating (\ref{s.13}) once, we have 
\bea\label{s.14} R^{(1)}(h^{(1)})^k\left((k+1)!+\frac{(k-1)(k+1)!}{(k+3)}\right)+k(k+1)!R(h^{(1)})^{k-1}h^{(2)}+R_4(h)=0,\eea
where each term of $R_{4}(h)$ contains $h^{m}\;(1\leq m \leq k$) as a factor. Now using (\ref{s.12}) to (\ref{s.14}), we have 
\bea\label{s.15}\left((k+1)!+\frac{4(k+1)!}{(k+3)}+\frac{4(k+1)!}{(k+3)}\right)R^{(1)}(h^{(1)})^k+R_5(h)=0,\eea
where each term of $R_{5}(h)$ contains $h^{m}\;(1\leq m \leq k$) as a factor. If $z_1$ is a simple zero of $h$, then from (\ref{s.15}), we observe that $R^{(1)}(z_1)=0$. Finally $h$ has only finitely many zeros, which is a contradiction. 
\end{proof}

\section{\bf{Proof of Theorem  \ref{t2.2}}}
\begin{proof} Let $f$ be a meromorphic solution of the Eq. (\ref{2.d}) such that $\rho(f)=+\infty$.
Then (\ref{lem1}) gives $m=2$ and so from (\ref{2.d}), we have
\bea\label{tp1} (Rf^{(k)}+\iota f)(Rf^{(k)}-\iota f)=Qe^{\alpha},\eea
which shows that $f$ has only finitely many poles. Now in view of (\ref{tp1}), we may assume that
\bea\label{tp2} \left\{\begin{array}{clcr} Rf^{(k)}+\iota f=Q_1e^{P_1},\\
Rf^{(k)}-\iota f=Q_2e^{P_2},\end{array}\right.\eea
where $Q_1$ and $Q_2$ are non-zero rational functions such that $Q_1Q_2\equiv Q$, $P_1$ and $P_2$ are
entire functions such that $P_1+P_2=\alpha$. Solving (\ref{tp2}) for $f^{(k)}$ and $f$, we obtain
\bea\label{tp3} f^{(k)}=\frac{Q_1e^{P_1}+Q_2e^{P_2}}{2R}\eea
and
\bea\label{tp3a} f=\frac{Q_1e^{P_1}-Q_2e^{P_2}}{2\iota}.\eea

Since $P_1+P_2=\alpha$, from (\ref{tp3a}), we deduce that $P_1$ and $P_2$ are transcendental functions. Clearly $P_1^{(1)}$and $P_2^{(1)}$ are also transcendental functions. Now differentiating (\ref{tp3a}) $k-$times, we get
\bea\label{tp4} f^{(k)}=\frac{\left(Q_1(P_1^{(1)})^k+\tilde{P}_{k-1}(P_1^{(1)})\right)e^{P_1}-\left(Q_2(P_2^{(1)})^k+\tilde{Q}_{k-1}(P_2^{(1)})\right)e^{P_2}}{2\iota},\eea
where $\tilde{P}_{k-1}(P_1^{(1)})$ and $\tilde{Q}_{k-1}(P_2^{(1)})$ are differential polynomials in $P_1^{(1)}$ and $P_2^{(1)}$ of degree $k-1$ with coefficients respectively $Q_1$, $Q_1^{(1)}$,$\ldots$, $Q_1^{(k)}$ and $Q_2, Q_2^{(1)},\ldots, Q_2^{(k)}$.
Then (\ref{tp3}) and (\ref{tp4}) give 
\bea\label{tp5} \left(RQ_1(P_1^{(1)})^k+R\tilde{P}_{k-1}(P_1^{(1)})-\iota Q_1\right)e^{P_1}=\left(RQ_2(P_2^{(1)})^k+R\tilde{Q}_{k-1}(P_2^{(1)})+\iota Q_2\right)e^{P_2}.\eea

If possible suppose $RQ_1(P_1^{(1)})^k+R\tilde{P}_{k-1}(P_1^{(1)})-\iota Q_1\equiv 0$. Then 
\[RQ_1(P^{(1)})^{k-1}P^{(1)}\equiv -R\tilde{P}_{k-1}(P^{(1)})+\iota Q_1\]
and so using Lemma \ref{l3.1}, we get $m(r,P^{(1)})=S(r,P^{(1)})$, i.e., $T(r,P^{(1)})=S(r,P^{(1)})$, which is impossible. Hence $RQ_1(P_1^{(1)})^k+R\tilde{P}_{k-1}(P_1^{(1)})-\iota Q_1\not\equiv 0$. 

Similarly we can prove that $RQ_2(P_2^{(1)})^k+R\tilde{Q}_{k-1}(P_2^{(1)})+\iota Q_2\not\equiv 0$. Since $P_1+P_2=\alpha$, we have $P_2^{(1)}=-P_1^{(1)}+\alpha^{(1)}$. Therefore from (\ref{tp5}), we get 
\beas\label{tp6}&&\left(RQ_1(P_1^{(1)})^k+R\tilde{P}_{k-1}(P_1^{(1)})-\iota Q_1\right)e^{2P_1-\alpha}\nonumber\\
=&&RQ_2(-P_1^{(1)}+\alpha^{(1)})^k+R\tilde{Q}_{k-1}(-P_1^{(1)}+\alpha^{(1)})+\iota Q_2\eeas
and so by Lemma \ref{l2}, we get $T(r,e^{2P_1-\alpha})=S(r,e^{2P_1-\alpha})$, which is a contradiction.
\end{proof}

\section {{\bf Proof of Theorem \ref{t2.3}}} 
\begin{proof} Let $f$ be a meromorphic solution of the Eq. (\ref{2.e}). Now by Theorem \ref{t2.2}, we have $\rho(f)<+\infty$. We rewrite the Eq. (\ref{2.e}) as follows
\bea\label{pt1} \left(e^{-\alpha/m}f\right)^m+\left(Re^{-\alpha/m}f^{(k)}\right)^m=1,\eea
which shows that $f$ has no poles. Let us take $F=e^{-\alpha/m}f$ and $G=Re^{-\alpha/m}f^{(k)}$. If $F$ and $G$  are non-constants, then proceeding in the same way as done in the proof of Lemma \ref{l3}, we can prove that $2/m\geq 1$, i.e., $m\leq 2$. Since $m\geq 2$, we get $m=2$.
Now we consider the following cases.

\smallskip
{\bf Case 1.} Let $m\geq 3$. In this case $F$ and $G$ must be constant. Let $F=c_1\in \mathbb{C}\setminus  \{0\}$ and $G=c_2\in \mathbb{C}\backslash \{0\}$. Then $f=c_1e^{\alpha/m}$ and $f^{(k)}=\frac{c_2}{R}e^{\alpha/m}$. Also from (\ref{pt1}), we have $c_1^m+c_2^m=1$. Now differentiating both sides of $f=c_1e^{\alpha/m}$ $k$-times, we get
\bea\label{pt2}f^{(k)}=c_1\left(\left(\alpha^{(1)}/m\right)^k+\tilde{P}_{k-1}\left(\alpha^{(1)}/m\right)\right)e^{\alpha/m},\eea
where $\tilde{P}_{k-1}\left(\alpha^{(1)}/m\right)$ is a differential polynomial in $\alpha_1^{(1)}/m$ of degree $k-1$ with constant coefficients. Since $f^{(k)}=\frac{c_2}{R}e^{\alpha/m}$, from (\ref{pt2}), we obtain
\bea\label{pt3}c_1R\left(\left(\alpha^{(1)}/m\right)^k+\tilde{P}_{k-1}\left(\alpha^{(1)}/m\right)\right)=c_2,\eea
which shows that $R$ reduces to a constant and $\alpha$ is a polynomial of degree at most one. Let $\alpha(z)=az+b$. Then (\ref{pt3}) gives $c_1R(a/m)^k=c_2$, which shows that $a\neq 0$. Since $c_1^m+c_2^m=1$, it follows that $c_1^m\left(R^m(a/m)^{km}+1\right)=1$. Finally, we have $f(z)=ce^{\frac{az+b}{m}}$,
 where $R$ reduces to a constant and $a(\neq 0), b, c(\neq 0)$ are constants such that $c^m\left(R^m(a/m)^{km}+1\right)=1$.

\smallskip
{\bf Case 2.} Let $m=2$. According to Theorem 4 \cite{FG2}, we find that the equation (\ref{pt1}) has entire solutions of the from $e^{-\alpha/2}f=\sin h$ and $Re^{-\alpha/2}f^{(k)}=\cos h$, where $h$ is an entire function. Since $\rho(f)<+\infty$, it follows that $h$ is a polynomial. Therefore
\bea\label{pt4} f=\frac{e^{P_1}-e^{P_2}}{2\iota}\eea
and
\bea\label{pt5} f^{(k)}=\frac{e^{P_1}+e^{P_2}}{2 R},\eea
where $P_1=\frac{\alpha}{2}+\iota h$ and $P_2=\frac{\alpha}{2}-\iota h$. Now differentiating (\ref{pt4}) $k$-times, we get
\bea\label{pt6}f^{(k)}=\frac{\left(\left(P_1^{(1)}\right)^k+\tilde{P}_{k-1}\left(P_1^{(1)}\right)\right)e^{P_1}-\left(\left(P_2^{(1)}\right)^k+\tilde{P}_{k-1}\left(P_2^{(1)}\right)\right)e^{P_2}}{2\iota},\eea
where $\tilde{P}_{k-1}(g)$ is a differential polynomial in $g$ of degree $k-1$ with constant coefficients.
Then (\ref{pt5}) and (\ref{pt6}) give 
\bea\label{pt7} \left(R\left(\left(P_1^{(1)}\right)^k+\tilde{P}_{k-1}\left(P_1^{(1)}\right)\right)-\iota \right)e^{P_1}=\left(R\left(\left(P_2^{(1)}\right)^k+\tilde{P}_{k-1}\left(P_2^{(1)}\right) \right)+\iota \right)e^{P_2}.\eea

We consider the following two sub-cases:

\smallskip
{\bf Sub-case 2.1.} Let $R\big((P_1^{(1)})^k+\tilde{P}_{k-1}(P_1^{(1)})\big)\not\equiv \iota$. Then  $R\big((P_2^{(1)})^k+\tilde{P}_{k-1}(P_2^{(1)})\big)\not\equiv -\iota$ and so from (\ref{pt7}), we get
\bea\label{pt8} \left(R\left(\left(P_1^{(1)}\right)^k+\tilde{P}_{k-1}\left(P_1^{(1)}\right)\right)-\iota \right)e^{2\iota h}=R\left(\left(P_2^{(1)}\right)^k+\tilde{P}_{k-1}\left(P_2^{(1)}\right) \right)+\iota.\eea

If $h$ is non-constant, then using Lemma \ref{l2} to (\ref{pt8}), we get a contradiction. Hence $e^{\iota h}=d$, a constant. Clearly $P_1^{(1)}=P_2^{(1)}=\frac{\alpha^{(1)}}{2}$. Now from (\ref{pt8}), we get $d^2\neq \pm 1$ and 
\bea\label{pt9}R\left(\left(\frac{\alpha^{(1)}}{2}\right)^k+\tilde{P}_{k-1}\left(\frac{\alpha^{(1)}}{2}\right)\right)=\iota \frac{d^2+1}{d^2-1},\eea
which shows that $R$ reduces to a constant and $\alpha$ is a polynomial of degree one. Let $\alpha(z)=az+b$, where $a(\neq 0),b$ are constants. Then (\ref{pt9}) gives $d^2(R(a/2)^k-\iota)=R(a/2)^k+\iota$. Finally from (\ref{pt4}), we have 
\[f(z)=\frac{d^2-1}{2\iota d}e^{\frac{az+b}{2}},\]
where $R$ reduces to a constant, $a(\neq 0), b, d(\neq 0)\in\mathbb{C}$ such that $d^2(R(a/2)^k-\iota)=R(a/2)^k+\iota$.

\smallskip
{\bf Sub-case 2.2.} Let $R\big((P_1^{(1)})^k+\tilde{P}_{k-1}(P_1^{(1)})\big)\equiv \iota$. Then  $R\big((P_2^{(1)})^k+\tilde{P}_{k-1}(P_2^{(1)})\big)\equiv -\iota$. In this case also $R$ reduces to a constant and both $P_1$ and $P_2$ are polynomials of degree one. Therefore
\bea\label{pt10}R\left(\frac{\alpha^{(1)}}{2}+\iota h^{(1)}\right)^k=\iota\eea
and 
\bea\label{pt11}R\left(\frac{\alpha^{(1)}}{2}-\iota h^{(1)}\right)^k=-\iota.\eea

It is easy to verify from (\ref{pt10}) and (\ref{pt11}) that $\deg(\alpha)\leq 1$ and $\deg(h)\leq 1$. If $\deg(h)=0$, then from (\ref{pt10}) and (\ref{pt11}), we immediately get a contradiction. Hence $\deg(h)=1$. Again if $\deg(\alpha)=0$, then from (\ref{pt10}) and (\ref{pt11}), we deduce that $k$ is odd. Let $\alpha(z)=a_1z+a_2$ and $h(z)=b_1z+b_2$, where $a_1,a_2, b_1(\neq 0),b_2$ are constants.
Finally from (\ref{pt4}), we have 
\[f(z)=e^{\frac{a_1z+a_2}{2}}\sin (b_1z+b_2),\]
where $R$ reduces to a constant, $a_1,a_2, b_1(\neq 0),b_2$ are constants such that $R\left(\frac{a_1}{2}+\iota b_1\right)^k=\iota$ and $R\left(\frac{a_1}{2}-\iota b_1\right)^k=-\iota$. Furthermore if $a_1=0$, then $k$ is odd.
\end{proof}

\section {{\bf Proof of Theorem \ref{t2.4}}} 

\begin{proof}
Let $f$ be a transcendental meromorphic solution of the Eq. (\ref{2.d}). Clearly from (\ref{2.d}), we deduce that $f$ has only finitely many poles. Also by Theorem \ref{t2.2}, we have $\rho(f)<+\infty$.
Now we consider the following two cases.

\smallskip
{\bf Case 1.} Let $e^{-\alpha/m}f$ be a rational function. Let $e^{-\alpha/m}f=R_1$, where $R_1(\not\equiv 0)$ be a rational function. Then 
\bea\label{tpa}f=R_1e^{\alpha/m}.\eea

Since $f$ is transcendental, it follows that $\alpha$ is a non-constant polynomial. Let 
\[g=\frac{R_1^{(1)}}{R_1}+\frac{1}{m}\alpha^{(1)}.\]

Now using Lemma \ref{l3.2} to (\ref{tpa}), we get
\bea\label{tpb} f^{(k)}=\left(R_1e^{\alpha/m}\right)^{(k)}=\left(g^k+R_{k-1}(g)\right)R_1e^{\alpha/m},\eea
where $R_{k-1}(g_1)$ is a differential polynomial with constant coefficients of degree $k-1$. 
Therefore using (\ref{tpa}) and (\ref{tpb}) to (\ref{2.d}), we get
\bea\label{tpc}(g^k+R_{k-1}(g))^m=\frac{Q-R_1^m}{(RR_1)^m}.\eea

Now letting $|z|\rightarrow \infty$ on the both sides of (\ref{tpc}), one can easily conclude that 
\[mk\deg(\alpha^{(1)})=\deg(Q-R_1^m)-m\deg(RR_1).\]

Finally, we have $f=R_1e^{\alpha/m}$, where $R_1(\not\equiv 0)$ is a rational function such that $mk\deg(\alpha^{(1)})=\deg(Q-R_1^m)-m\deg(RR_1)$ and $R_1^m+R^m e^{-\alpha}\left((R_1e^{\alpha/m})^{(k)}\right)^m=Q$.

\medskip
{\bf Case 2.} Let $e^{-\alpha/m}f$ be transcendental.
Then from (\ref{lem1}), we get $m=2$ and so from (\ref{2.d}), we obtain
\bea\label{tp1} (Rf^{(k)}+\iota f)(Rf^{(k)}-\iota f)=Qe^{\alpha}.\eea

Therefore in view of (\ref{tp1}), we may assume that
\bea\label{tp2} \left\{\begin{array}{clcr} Rf^{(k)}+\iota f=Q_1e^{P_1},\\
Rf^{(k)}-\iota f=Q_2e^{P_2},\end{array}\right.\eea
where $Q_1$ and $Q_2$ are non-zero rational functions such that $Q_1Q_2\equiv Q$, $P_1$ and $P_2$ are polynomials such that $P_1+P_2=\alpha$. Solving (\ref{tp2}) for $f^{(k)}$ and $f$, we obtain
\bea\label{tp3} f^{(k)}=\frac{Q_1e^{P_1}+Q_2e^{P_2}}{2R}\eea
and
\bea\label{tp3a} f=\frac{Q_1e^{P_1}-Q_2e^{P_2}}{2\iota}.\eea

Let 
\[g_i=\frac{Q_i^{(1)}}{Q_i}+P_i^{(1)},\]
where $i=1,2$. Now using Lemma \ref{l3.2}, we get for $i=1,2$
\bea\label{tp7} (Q_ie^{P_i})^{(k)}=R_i(g_i)Q_ie^{P_i},\eea
where
\bea\label{tp8} \tilde R_i(g_i)=g_i^{k}+\frac{k(k-1)}{2}g_i^{k-2}g_i^{(1)}+a_{k}g_i^{k-3}g_i^{(2)}+b_{k}g_i^{k-4}(g_i^{(1)})^{2}+\tilde P_{k-3}(g_i),\eea 
$a_{k}=\frac{1}{6}k(k-1)(k-2)$, $b_{k}=\frac{1}{8}k(k-1)(k-2)(k-3)$ and $\tilde P_{k-3}(g_i)$ is a differential polynomial with constant coefficients, which vanishes identically for $k\leq 3$ and has degree $k-3$ when $k>3$. 

Now differentiating (\ref{tp3a}) $k-$times, we get
\bea\label{tp9} f^{(k)}=\frac{\tilde R_1(g_1)Q_1e^{P_1}-\tilde R_2(g_2)Q_2e^{P_2}}{2\iota}.\eea

Then (\ref{tp3}) and (\ref{tp9}) give 
\bea\label{tp10} Q_1\left(R\tilde R_1(g_1)-\iota \right)e^{P_1}=Q_2\left(R\tilde R_2(g_2)+\iota \right)e^{P_2}.\eea

We consider the following two sub-cases:

\smallskip
{\bf Sub-case 2.1.} Let $R\tilde R_1(g_1)-\iota\not\equiv 0$. Then from (\ref{tp10}), we have $R\tilde R_2(g_2)+\iota\not\equiv 0$ and
\bea\label{tp10a} Q_1\left(R\tilde R_1(g_1)-\iota \right)e^{P_1-P_2}=Q_2\left(R\tilde R_2(g_2)+\iota \right).\eea

If $P_1-P_2$ is non-constant, then using Lemma \ref{l2} to (\ref{tp10a}), we get a contradiction. Hence $P_1-P_2=d$, a constant. Since $P_1+P_2=\alpha$, we have $P_1=\frac{1}{2}(\alpha+d)$ and $P_2=\frac{1}{2}(\alpha-d)$. Now from (\ref{tp3a}), we get
\bea\label{tp10b} f=\frac{d_1^2Q_1-Q_2}{2\iota d_1}e^{\alpha/2},\eea
where $d_1=e^{d/2}$. Since $f$ is non-constant, it follows that $d_1^2Q_1-Q_2\not\equiv 0$ and $\alpha$ is non-constant.

\smallskip
{\bf Sub-case 2.2.} Let $R\tilde R_1(g_1)-\iota\equiv 0$. Then from (\ref{tp10}), we also have $R\tilde R_2(g_2)+\iota\equiv 0$. Therefore from (\ref{tp8}), we have respectively
\bea\label{tp11} g_1^{k}+\frac{k(k-1)}{2}g_1^{k-2}g_1^{(1)}+a_{k}g_1^{k-3}g_1^{(2)}+b_{k}g_1^{k-4}(g_1^{(1)})^{2}+\tilde P_{k-3}(g_1)\equiv \frac{\iota}{R}\eea
and 
\bea\label{tp12} g_2^{k}+\frac{k(k-1)}{2}g_2^{k-2}g_2^{(1)}+a_{k}g_2^{k-3}g_2^{(2)}+b_{k}g_2^{k-4}(g_2^{(1)})^{2}+\tilde P_{k-3}(g_2)\equiv -\frac{\iota}{R}.\eea

Following sub-cases are immediate.

\smallskip
{\bf Sub-case 2.2.1.} Let $R$ be a polynomial. Now letting $|z|\rightarrow \infty$ on the both sides of (\ref{tp11}), one can easily conclude that $R$ is a non-zero constant, say $A$. Again letting $|z|\rightarrow \infty$ on the both sides of (\ref{tp11}) and (\ref{tp12}), we deduce that 
\[(P_1^{(1)})^k\equiv \frac{\iota}{A}\;\;\text{and}\;\;(P_2^{(1)})^k\equiv -\frac{\iota}{A}.\]

On integration, we have
$P_1(z)=a_1z+b_1$ and $P_2(z)=a_2z+b_2$, where $a_1^k=\frac{\iota}{A}$, $a_2^k=-\frac{\iota}{A}$, $b_1$ and $b_2$ are constants. Clearly $\alpha=(a_1+a_2)z+b_1+b_2$. 
On the other hand from (\ref{tp3}) and (\ref{tp3a}), one can easily that both $Q_1$ and $Q_2$ are polynomials. Note that
\beas \left(Q_1e^{P_1}\right)^{(k)}&=&\left(Q_1(P_1^{(1)})^k+\tilde{P}_{k-1}(P_1^{(1)})\right)e^{P_1}\\&=&
\left(Q_1^{(k)}+ka_1Q_1^{(k-1)}+\cdots+ka_1^{k-1}Q_1^{(1)}+a_1^kQ_1\right)e^{P_1},\eeas
i.e.,
\bea\label{tp13}Q_1(P_1^{(1)})^k+\tilde{P}_{k-1}(P_1^{(1)})=Q_1^{(k)}+ka_1Q_1^{(k-1)}+\cdots+ka_1^{k-1}Q_1^{(1)}+a_1^kQ_1.\eea

Similarly we have
\bea\label{tp14} Q_2(P_2^{(1)})^k+\tilde{Q}_{k-1}(P_2^{(1)})=Q_2^{(k)}+ka_2Q_2^{(k-1)}+\cdots+ka_2^{k-1}Q_2^{(1)}+a_2^kQ_2.\eea

Now using (\ref{tp13}) and (\ref{tp14}) into (\ref{tp4}) and then comparing with (\ref{tp3}), we conclude that both $Q_1$ and $Q_2$ are constants. Finally from (\ref{tp3a}), we may assume that
\[f(z)=\frac{Q_1e^{a_1z+b_1}-Q_2e^{a_2z+b_2}}{2\iota},\]
where $R\equiv A(\text{constant})$ such that $Aa_1^k=\iota$, $Aa_2^k=-\iota$, $Q_1$ and $Q_2$ are constants such that $Q_1Q_2=Q$, $b_1$ and $b_2$ are constants such that $\alpha(z)=(a_1+a_2)z+b_1+b_2$. Furthermore if $\alpha$ is constant, i.e., if $a_1+a_2=0$, then $k$ is odd.

\smallskip
{\bf Sub-case 2.2.2.} Let $R$ be a non-constant rational function other than polynomial. If possible suppose $\deg(R)>0$. Note that $P_1^{(1)}\not\equiv 0$ and $P_2^{(1)}\not\equiv 0$. Now letting $|z|\rightarrow \infty$ on the both sides of (\ref{tp11}) and (\ref{tp12}), we immediately get a contradiction. Hence $\deg(R)\leq 0$. We consider following sub-cases.

\smallskip
{\bf Sub-case 2.2.2.1.} Let $\deg(R)=0$. Now letting $|z|\rightarrow \infty$ on the both sides of (\ref{tp11}) and (\ref{tp12}), we deduce that both $P_1$ and $P_2$ are polynomials of degree one. Let 
\[P_1(z)=a_1z+b_1\;\;\text{and}\;\;P_2(z)=a_2z+b_2,\]
 where $a_1(\neq 0), a_2(\neq 0), b_1,b_2$ are constants. Since $P_1+P_2=\alpha$, we have $\alpha(z)=(a_1+a_2)z+b_1+b_2$. Now using (\ref{tp13}) and (\ref{tp14}) into (\ref{tp4}) and then comparing with (\ref{tp3}), we have respectively
\bea\label{ttp1}R\left(Q_1^{(k)}+ka_1Q_1^{(k-1)}+\cdots+ka_1^{k-1}Q_1^{(1)}+a_1^kQ_1\right)=\iota Q_1\eea
and
\bea\label{ttp2}R\left(Q_2^{(k)}+ka_2Q_2^{(k-1)}+\cdots+ka_2^{k-1}Q_2^{(1)}+a_2^kQ_2\right)=-\iota Q_2.\eea

Adding (\ref{ttp1}) and (\ref{ttp2}), we get
\bea\label{tp14a}&&\frac{Q_1^{(k)}}{Q_1}+\frac{Q_2^{(k)}}{Q_2} +k\left(a_1\frac{Q_1^{(k-1)}}{Q_1}+a_2\frac{Q_2^{(k-1)}}{Q_2}\right)+\cdots\nonumber\\
&&+k\left(a_1^{k-1}\frac{Q_1^{(1)}}{Q_1}+a_2^{k-1}\frac{Q_2^{(1)}}{Q_2}\right)+a_1^k+a_2^{k}\equiv 0.\eea

Letting $|z|\rightarrow \infty$ on the both sides of (\ref{tp14a}), we deduce that $a_1^k+a_2^{k}=0$. Let $a_2=ta_1$, where $t^k=-1$. Now (\ref{tp14a}) gives
\bea\label{tp14b}\frac{Q_1^{(k)}}{Q_1}+\frac{Q_2^{(k)}}{Q_2} +ka_1\left(\frac{Q_1^{(k-1)}}{Q_1}+t\frac{Q_2^{(k-1)}}{Q_2}\right)+\cdots+ka_1^{k-1}\left(\frac{Q_1^{(1)}}{Q_1}+t^{k-1}\frac{Q_2^{(1)}}{Q_2}\right)\equiv 0.\eea

We now want to prove that 
\bea\label{tp14bb}\frac{Q_1^{(1)}}{Q_1}+t^{k-1}\frac{Q_2^{(1)}}{Q_2}\equiv 0.\eea

If $k=1$, then from (\ref{tp14b}), we have have $\frac{Q_1^{(1)}}{Q_1}+\frac{Q_2^{(1)}}{Q_2}\equiv 0$ and so (\ref{tp14bb}) is true. Next let $k\geq 2$. If possible suppose $\frac{Q_1^{(1)}}{Q_1}+t^{k-1}\frac{Q_2^{(1)}}{Q_2}\not\equiv 0$. Then (\ref{tp14b}) gives
\bea\label{tp14c}\frac{\frac{Q_1^{(k)}}{Q_1}+\frac{Q_2^{(k)}}{Q_2}}{\frac{Q_1^{(1)}}{Q_1}+t^{k-1}\frac{Q_2^{(1)}}{Q_2}}+ka_1\frac{\frac{Q_1^{(k-1)}}{Q_1}+t\frac{Q_2^{(k-1)}}{Q_2}}{\frac{Q_1^{(1)}}{Q_1}+t^{k-1}\frac{Q_2^{(1)}}{Q_2}}+\cdots+ka_1^{k-1}\equiv 0.\eea

It is easy to verify that 
\[\deg\left(\left(\frac{Q_1^{(k)}}{Q_1}+\frac{Q_2^{(k)}}{Q_2}\right)/\left(\frac{Q_1^{(1)}}{Q_1}+t^{k-1}\frac{Q_2^{(1)}}{Q_2}\right)\right)<0,\]
\[\deg\left(\left(\frac{Q_1^{(k-1)}}{Q_1}+t\frac{Q_2^{(k-1)}}{Q_2}\right)/\left(\frac{Q_1^{(1)}}{Q_1}+t^{k-1}\frac{Q_2^{(1)}}{Q_2}\right)\right)<0, \;\text{etc}.\]

 Now letting $|z|\rightarrow \infty$ on the both sides of (\ref{tp14c}), we get $a_1=0$, which is impossible. Hence 
\[\frac{Q_1^{(1)}}{Q_1}+t^{k-1}\frac{Q_2^{(1)}}{Q_2}\equiv 0\]
for $k\geq 2$. Thus (\ref{tp14bb}) is true for $k\geq 1$. Since $t^k=-1$, from (\ref{tp14bb}), we get 
\[\frac{Q_1^{(1)}}{Q_1}+\frac{Q_2^{(1)}}{Q_2}\equiv 0\]
and so $(Q_1Q_2)^{(1)}\equiv 0$, i.e., $Q^{(1)}\equiv 0$. Clearly $Q$ is a constant. Since $R$ is non-constant, from (\ref{ttp1}) and (\ref{ttp2}), we conclude that both $Q_1$ and $Q_2$ are non-constants. Finally from (\ref{tp3a}), we may assume that
\[f(z)=\frac{Q_1(z)e^{a_1z+b_1}-Q_2(z)e^{a_2z+b_2}}{2\iota},\]
where $Q\equiv B(\text{constant})$, $Q_1$, $Q_2$ and $R$ are non-constant rational functions such that $Q_1Q_2=B$, $\deg(R)=0$ and $a_1(\neq 0), a_2(\neq 0)$, $b_1$ and $b_2$ are constants such that $a_1^k+a_2^k=0$ and $\alpha(z)=(a_1+a_2)z+b_1+b_2$. Furthermore if $\alpha$ is a constant, then $k$ is odd.

\smallskip
{\bf Sub-case 2.2.2.2.} Let $\deg(R)<0$. Now letting $|z|\rightarrow \infty$ on the both sides of (\ref{tp11}) and (\ref{tp12}), we get $k\deg(P_1^{(1)})=-\deg(R)$ and $k\deg(P_2^{(1)})=-\deg(R)$. Adding (\ref{tp11}) and (\ref{tp12}), we get
\bea\label{tp15}&&\left(\frac{Q_1^{(1)}}{Q_1}\right)^k+\left(\frac{Q_2^{(1)}}{Q_2}\right)^k+k\left(\left(\frac{Q_1^{(1)}}{Q_1}\right)^{k-1}P_1^{(1)}+\left(\frac{Q_2^{(1)}}{Q_2}\right)^{k-1}P_2^{(1)}\right)+\ldots\nonumber\\
&&+\left(\left(P_1^{(1)}\right)^k+\left(P_2^{(1)}\right)^k\right)+\tilde P_{k-1}(g_1)+\tilde P_{k-1}(g_2)\equiv 0,\eea
where $\tilde P_{k-1}(g_i)$ is a differential polynomial in $g_i$ with constant coefficients. We see that the highest powers of $z$ in the left hand side of (\ref{tp15}) appear in $\left(P_1^{(1)}\right)^k+\left(P_2^{(1)}\right)^k$. Therefore from (\ref{tp15}), we must have $\left(P_1^{(1)}\right)^k+\left(P_2^{(1)}\right)^k\equiv 0$, i.e., $P_1^{(1)}=tP_2^{(1)}$, where $t^k=-1$. Since $P_1^{(1)}+P_2^{(1)}=\alpha^{(1)}$, we have 
\bea\label{tp15a}(t+1)P_2^{(1)}=\alpha^{(1)}.\eea

Since $P_1^{(1)}=tP_2^{(1)}$, it follows that $\deg(P_1)=\deg(P_2)\geq 1$ and so $P_1^{(1)}\not\equiv 0$. Consequently from (\ref{tp15a}), we conclude that if $\alpha$ is a constant, then $t=-1$ and so $k$ is odd. Let $P_2=P$. Since $P_1^{(1)}=tP_2^{(1)}$, we assume that $P_1=tP+c$, where $c$ is a constant. 
Finally from (\ref{tp3a}), we may assume that
\bea\label{tp16}f(z)=\frac{Q_1(z)e^{tP(z)+c}-Q_2(z)e^{P(z)}}{2\iota},\eea
where $P$ is a non-constant polynomial, $Q_1$, $Q_2$ and $R$ are non-constant rational functions such that $k\deg\left(P^{(1)}\right)=-\deg(R)$, $Q_1Q_2=Q$, $\deg(R)<0$, $c$ and $t$ are constants such that $t^k=-1$ and $(t+1)P^{(1)}=\alpha^{(1)}$. If $\alpha$ is a constant, then $t=-1$ and $k$ is odd.\par

Furthermore assume that all the zeros of $R$ have multiplicities at most $k-1$. Differentiating (\ref{tp16}) $k-$times, we get
\bea\label{tp17}f^{(k)}(z)=\frac{\tilde Q_k(z)e^{tP(z)+c}-\tilde R_k(z)e^{P(z)}}{2\iota},\eea
where $\tilde Q_k(z)$ and $\tilde R_k(z)$ are rational functions such that 
\bea\label{tp18}\tilde Q_k=Q_1^{(k)}+ktP^{(1)}Q_1^{(k-1)}+\ldots+\hat Q_k(P^{(1)})Q_1\eea
and 
\bea\label{tp19}\tilde R_k=Q_2^{(k)}+kP^{(1)}Q_1^{(k-1)}+\ldots+\hat Q_k(P^{(1)})Q_2,\eea
$\hat Q_k(P^{(1)})$ is a differential polynomial in $P^{(1)}$ of degree $k$ with constant coefficients.

Now using (\ref{tp18}) and (\ref{tp19}) into (\ref{tp17}) and then comparing with (\ref{tp3}) (taking $P_1=tP+c$ and $P_2=P$), we get respectively
\bea\label{tp18a}R\left(Q_1^{(k)}+ktP^{(1)}Q_1^{(k-1)}+\ldots+\hat Q_k(P^{(1)})Q_1\right)=\iota Q_1\eea
and 
\bea\label{tp19a}R\left(Q_2^{(k)}+kP^{(1)}Q_1^{(k-1)}+\ldots+\hat Q_k(P^{(1)})Q_2\right)=-\iota Q_2.\eea

Since all the zeros of $R$ have multiplicities at most $k-1$, from (\ref{tp18a}) and (\ref{tp19a}) one can easily deduce that both $Q_1$ and $Q_2$ are polynomials. Therefore $Q$ is also a polynomial. If $k=1$, then by Theorem B, we get that $Q$ is a constant. Clearly both $Q_1$ and $Q_2$ are non-zero constants. Also from (\ref{tp18a}), we get $R=1/P^{(1)}$.
\end{proof}

\section{{\bf Application:}}
The Fermat-type complex differential equation
\beas f^m+\big(Rf^{(k)}\big)^n=Qe^{\alpha},\eeas
 play an important role in the theory of nonlinear differential polynomials and the growth analysis of entire and meromorphic functions. Such equations naturally occur in the study of value-distribution problems, where they help characterize the interaction between a function and its higher-order derivatives. They also arise in complex dynamical systems, particularly in describing exponential forcing and growth-controlled trajectories. In mathematical physics, these equations appear as complex analogues of nonlinear balance laws governing oscillatory and wave-type phenomena. Moreover, they have applications in the geometric theory of holomorphic curves and complex differential geometry, where differential constraints of Fermat type determine curvature and immersion properties. Understanding the existence or non-existence of entire solutions therefore yields meaningful insight into the analytic structure, uniqueness behaviour, and asymptotic properties of complex differential systems.

\vspace{0.1in}
{\bf Compliance of Ethical Standards:}\par

{\bf Conflict of Interest.} The authors declare that there is no conflict of interest regarding the publication of this paper.\par

{\bf Data availability statement.} Data sharing not applicable to this article as no data sets were generated or analysed during the current study.

\end{document}